\documentclass{amsart}
\usepackage{graphicx}

\begin{document}
%
%
% THEOREM Environments (Examples)-----------------------------------------
%

 \newtheorem{thm}{Theorem}
 \newtheorem{cor}[thm]{Corollary}
 \newtheorem{lem}[thm]{Lemma}
 \newtheorem{prop}[thm]{Proposition}
 \theoremstyle{definition}
 \newtheorem{defn}[thm]{Definition}
\newtheorem{assumption}{Assumption}
\newtheorem{ex}{Example}
 \theoremstyle{remark}
 \newtheorem{rem}[thm]{Remark}
 
 \numberwithin{equation}{section}
 \def\e{{\rm e}}
 \def\x{\mathbf{x}}
\def\F{\mathcal{F}}
\def\R{\mathbb{R}}
\def\T{\mathbf{T}}
\def\N{\mathbb{N}}
\def\K{\mathbf{K}}
\def\Q{\mathbf{Q}}
\def\M{\mathbf{M}}
\def\O{\mathbf{O}}
\def\C{\mathbb{C}}
\def\P{\mathbf{P}}
\def\Z{\mathbb{Z}}
\def\H{\mathcal{H}}
\def\A{\mathbf{A}}
\def\V{\mathbf{V}}
\def\AA{\overline{\mathbf{A}}}
\def\B{\mathbf{B}}
\def\c{\mathbf{C}}
\def\L{\mathbf{L}}
\def\H{\mathcal{H}}
\def\I{\mathbf{I}}
\def\Y{{\rm Y}}
\def\f{\mathbf{f}}
\def\a{\mathbf{a}}
\def\z{\mathbf{z}}
\def\d{\hat{d}}
\def\y{\mathbf{y}}
\def\v{\mathbf{v}}
\def\w{\mathbf{w}}
\def\b{\mathbf{b}}
\def\s{\mathcal{S}}
\def\cc{\mathcal{C}}
\def\co{{\rm co}\,}
\def\vol{{\rm vol}\,}
\def\om{\mathbf{\Omega}}
\def\dom{{\rm Dom}}
\def\ite{{\rm Int}\,}

%-------------------------------------------------------------------------
% editorial commands: to be inserted by the editorial office
%
%\firstpage{1}
%\volume{228}
%\Copyrightyear{2004}
%\DOI{003-0001}
%
%
%\seriesextra{Just an add-on}
%\seriesextraline{This is the Concrete Title of this Book\br H.E. R and S.T.C. W, Eds.}
%
% for journals:
%
%\firstpage{1}
%\issuenumber{1}
%\Volumeandyear{1 (2004)}
%\Copyrightyear{2004}
%\DOI{003-xxxx-y}
%\Signet
%\commby{inhouse}
%\submitted{March 14, 2003}
%\received{March 16, 2000}
%\revised{June 1, 2000}
%\accepted{July 22, 2000}
%
%
%
%---------------------------------------------------------------------------
%Insert here the title, affiliations and abstract:
%
\title{On convex optimization without convex representation}

%----------Author 1
\author{JB. Lasserre}
\address{LAAS-CNRS and Institute of Mathematics\\
University of Toulouse\\
LAAS, 7 avenue du Colonel Roche\\
31077 Toulouse C\'edex 4\\
France}
\email{lasserre@laas.fr}

%----------classification, keywords, date


\date{}
%----------additions
%\dedicatory{To my boss}
%%% ----------------------------------------------------------------------

\begin{abstract}
We consider the convex optimization problem
$\P:\min_\x \{f(\x)\,:\,\x\in\K\}$ where $f$ is convex continuously differentiable, and
$\K\subset\R^n$ is a compact convex set with representation
$\{\x\in\R^n\,:\:g_j(\x)\geq0, j=1,\ldots,m\}$ for some continuously differentiable
functions $(g_j)$. We discuss the case where the 
$g_j$'s are not all concave (in contrast with convex programming 
where they all are). In particular,
even if the $g_j$ are not concave, we consider the log-barrier function $\phi_\mu$
with parameter $\mu$, associated with $\P$, usually defined for concave functions $(g_j)$.
We then show that any limit point of any sequence $(\x_\mu)\subset\K$
of stationary points of $\phi_\mu$, $\mu\to0$,  is a Karush-Kuhn-Tucker point
of problem $\P$ and a global minimizer of $f$ on $\K$.
\end{abstract}

%%% ----------------------------------------------------------------------
\maketitle
%%% ----------------------------------------------------------------------
%\tableofcontents

\section{Introduction}~
%\baselineskip=1.5\baselineskip

%With $\R[x]$ being the ring of real polynomials in the variables $x_1,\ldots,x_n$, 
Consider the optimization problem
\begin{equation}
\label{defpb}
\P:\quad f^*:=\,\displaystyle\min_{\x}\:\{\:f(\x)\::\:\x\in \K\}.
\end{equation}
for some convex and continuously differentiable function $f:\R^n\to\R$,
and where the feasible set $\K\subset\R^n$ is defined by:
\begin{equation}
\label{setk}
\K:=\,\{\x\in\R^n\::\:g_j(\x)\geq 0,\quad j=1,\ldots,m\},
\end{equation}
for some continuously differentiable functions $g_j:\R^n\to\R$.
We say that $(g_j)$, $j=1,\ldots,m$, is a {\it representation} of $\K$.
When $\K$ is convex and the $(g_j)$ are concave we say that 
$\K$ has a convex representation.

In the literature, when $\K$ is convex $\P$ is referred to as a convex optimization problem 
and in particular, every local minimum of $f$ is a global minimum.
However, if on the one hand
{\it convex optimization} usually refers to minimizing a convex function on a convex set 
$\K$ without precising its representation $(g_j$) (see e.g. Ben-Tal and Nemirovsky \cite[Definition 5.1.1]{bental} or Bertsekas et al. \cite[Chapter 2]{bertsekas}), on the other hand
{\it convex programming} usually refers to the situation where 
the representation of $\K$ is also convex, i.e. when all the $g_j$'s are concave.
See for instance Ben-Tal and Nemirovski \cite[p. 335]{bental},
Berkovitz \cite[p. 179]{berkovitz},
Boyd and Vandenberghe \cite[p. 7]{boyd}, Bertsekas et al. \cite[\S 3.5.5]{bertsekas},
Nesterov and Nemirovski \cite[p. 217-218]{nesterov}, and Hiriart-Urruty \cite{jbhu}.
Convex programming is particularly interesting because under Slater's condition\footnote{Slater's condition holds if $g_j(\x_0)>0$ for some $\x_0\in\K$ and all $j=1,\ldots,m$.},
the standard Karush-Kuhn-Tucker (KKT) 
optimality conditions are not only necessary but also sufficient and
in addition, the concavity property of the $g_j$'s is used to prove 
convergence (and rates of convergence) of specialized algorithms.\\

To the best of our knowledge,
little is said in the literature for the specific case where 
$\K$ is convex but not necessarily its representation, that is, when the functions $(g_j)$ are {\it not} necessarily concave. It looks like outside the convex programming framework, all problems are treated the same. This paper is a companion paper to \cite{lasserre} where we proved that if 
the nondegeneracy condition 
\begin{equation}
\label{dege}
\forall j=1,\ldots,m:\quad 
\nabla g_j(\x)\neq0\quad \forall\x\in\K\:\mbox{with $g_j(\x)=0$}\end{equation}
holds, then 
$\x\in\K$ is a global minimizer of $f$ on $\K$ if and only if
$(\x,\lambda)$ is a KKT point for some $\lambda\in\R^m_+$.
This indicates that for convex optimization problems (\ref{defpb}),
and from the point of view of "first-order optimality conditions",
what really matters is the geometry of $\K$ rather than its representation.
Indeed, for {\it any} representation $(g_j)$ of $\K$ 
that satisfies the nondegeneracy condition (\ref{dege}),
there is a one-to-one correspondence between global minimizers and KKT points.

But what about from a computational viewpoint?
Of course, not all representations of $\K$ are
equivalent since the ability (as well as the efficiency)
of algorithms to obtain a KKT point of $\P$
will strongly depend on the representation $(g_j)$ of $\K$ which is used.
For example, algorithms that implement Lagrangian duality
would require the $(g_j)$ to be concave, those based on second-order methods would require
all functions $f$ and $(g_j)$ to be twice continuous differentiable, self-concordance
of a barrier function associated with a representation of $\K$ may or may not hold, etc.

When $\K$ is convex but not its representation $(g_j)$, several situations may
occur. 
In particular, the level set  $\{\x: g_j(\x)\geq a_j\}$ may be convex
for $a_j=0$ but not for some other values of $a_j>0$, in which case the $g_j$'s are 
not even quasiconcave on $\K$, i.e., one may say that $\K$ is convex {\it by accident}
for the value $\a=0$ of the parameter $\a\geq0$.
One might think that in this situation, algorithms that generate
a sequence of feasible points in the interior of $\K$ 
could run into problems to find a local minimum of $f$.  If the $-g_j$'s are all quasiconvex 
on $\K$, we say that we are in the generic convex case
because not only $\K$ but also all sets $\K_\a:=\{\x: g_j(\x)\geq\a_j, j=1,\ldots,m\}$ are convex. However, quasiconvex functions do not share some nice properties of the convex functions. In particular, (a) $\nabla g_j(\x)=0$ does not imply that $g_j$ reaches a local minimum at $\x$, (b) a local minimum is not necessarily global and (c), the sum of quasiconvex functions is not quasiconvex in general; see e.g. 
Crouzeix et al. \cite[p. 65]{crouzeix}. And so even in this case, 
for some minimization algorithms, convergence to a minimum of $f$ on $\K$
might be problematic.\\

So an interesting issue is to determine whether
there is an algorithm which converges to
a global minimizer of a convex function $f$ on $\K$, no matter if the representation of
$\K$ is convex or not. Of course, in view of \cite[Theorem 2.3]{lasserre}, a sufficient condition is that this algorithm provides
a sequence (or subsequence) of points $(\x_k,\lambda_k)\in\R^n\times\R^m_+$
converging to a KKT point of $\P$.\\

With $\P$ and a parameter $\mu>0$,
we associate the {\it log-barrier} function $\phi_\mu: \K\to\R\cup \{+\infty\}$
defined by
\begin{equation}
\label{phimu}
\x\mapsto \phi_\mu(\x)\,:=\,\left\{\begin{array}{rl}
f(\x)-\mu\,\displaystyle\sum_{j=1}^m\ln{g_j(\x)},&\mbox{if $g_j(\x)>0,\:\forall j=1,\ldots,m$}\\
+\infty,&\mbox{otherwise.}\end{array}\right.
\end{equation}
By a {\it stationary point} $\x\in\K$ of $\phi_\mu$, we mean a point
$\x\in\K$ with $g_j(\x)\neq0$ for all $j=1,\ldots,m$, and such that $\nabla\phi_\mu(\x)=0$.
Notice that in general and in contrast with the present paper,
$\phi_\mu$ (or more precisely $\psi_\mu:=\mu\phi_\mu$) is usually defined for convex problems $\P$ where all the $g_j$'s are concave; see e.g. Den Hertog \cite{hertog}
and for more details on the barrier functions and their properties, the interested reader is referred to G\"uler  \cite{guler0} and G\"uler and Tuncel \cite{guler}.

{\bf Contribution.} The purpose of this paper is to show 
that no matter which representation $(g_j)$ of a convex set $\K$ (assumed to be compact)
is used (provided it satisfies the nondegeneracy condition (\ref{dege})), 
any sequence of stationary points $(\x_\mu)$ of $\phi_\mu$, $\mu\to 0$, 
has the nice property that each of its accumulation points is a KKT point of $\P$ and hence, a global minimizer of $f$ on $\K$. Hence, to obtain the global minimum
of a convex function on $\K$ it is enough to 
minimize the log-barrier function for nonincreasing values of the parameter,
for any representation of $\K$ that satisfies the nondegeneracy condition (\ref{dege}).
Again and of course, the efficiency of the method will crucially depend on the representation of $\K$ which is used. For instance, in general $\phi_\mu$ will not have the self-concordance property, crucial for efficiency.

Observe that at first glance this result is a little surprising because as 
we already mentioned, there are examples of 
sets $\K_\a:=\{\x: g_j(\x)\geq a_j,\,j=1,\ldots,m\}$
which are non convex for every $0\neq \a\geq0$ but $\K:=\K_{0}$ is convex (by accident!) and (\ref{dege}) holds. So inside $\K$ the level sets of the $g_j$'s are not convex any more. Still, and even though the stationary points $\x_\mu$ of the associated 
log-barrier $\phi_\mu$ are inside $\K$, all converging subsequences of a sequence $(\x_\mu)$, $\mu\to 0$,
 will converge to some global minimizer $\x^*$ of $f$ on $\K$.
 In particular, if the global minimizer $\x^*\in\K$ is unique then the whole sequence
 $(\x_\mu)$ will converge. Notice that this happens even if the $g_j$'s are not log-concave,
 in which case $\phi_\mu$ may not be convex for all $\mu$ (e.g. if $f$ is linear).
 So what seems to really matter is the fact that
  as $\mu$ decreases, the convex function $f$ becomes more and more 
 important in $\phi_\mu$, and also that the functions $g_j$ which 
 matter in a KKT point $(\x^*,\lambda)$ are those for which 
 $g_j(\x^*)=0$ (and so with convex associated level set $\{\x:g_j(\x)\geq0\}$).

\section{Main result}
Consider the optimization problem
(\ref{defpb}) in the following context.
\begin{assumption}
\label{ass1}
The set $\K$ in (\ref{setk}) is convex and Slater's assumption holds.
Morover, the nondegeneracy condition 
\begin{equation}
\label{aa}
\nabla g_j(\x)\neq 0\quad\forall\,\x\in\K\:\mbox{such that}\:g_j(\x)=0,\end{equation}
holds for every $j=1,\ldots,m$.
\end{assumption}
Observe that when the $g_j$'s are concave then the nondegeneracy condition 
(\ref{aa}) holds automatically.
Recall that $(\x^*,\lambda)\in\K\times\R^m$ is a Karush-Kuhn-Tucker
(KKT) point of $\P$ if
\begin{itemize}
\item $\x\in\K$ and $\lambda\geq0$
\item $\lambda_jg_j(\x^*)=0$ for every $j=1,\ldots,m$
\item $\nabla f(\x^*)-\sum_{j=1}^m\lambda_j\nabla g_j(\x^*)=0$.
\end{itemize}
We recall the following result from \cite{lasserre}:
\begin{thm}[\cite{lasserre}]
\label{th-optimletters}
Let $\K$ be as in (\ref{setk}) and let Assumption \ref{ass1} hold.
Then $\x$ is a global minimizer of $f$ on $\K$ if and only
if there is some $\lambda\in\R^m_+$ such that $(\x,\lambda)$ is a KKT point of $\P$.
\end{thm}
The next result is concerned with the log-barrier $\phi_\mu$ in (\ref{phimu}).
\begin{lem}
\label{lemma1}
Let $\K$ in (\ref{setk}) be convex and compact and assume that Slater's condition holds. Then for every $\mu>0$ the log-barrier function $\phi_\mu$ in (\ref{phimu}) has at least one 
stationary point on $\K$ (which is a global minimizer of $\phi_\mu$ on $\K$).
\end{lem}
\begin{proof}
Let $f^*$ be the minimum of $f$ on $\K$ and let $\mu>0$ be fixed, arbitrary. We first show that $\phi_\mu(\x_k)\to\infty$ as 
$\x_k\to \partial\K$ (where $(\x_k)\subset\K$). Indeed, pick up an index $i$ such that
$g_i(\x_k)\to0$ as $k\to\infty$. Then $\phi_\mu(\x_k)\geq f^*-\mu\ln g_i(\x_k)-(m-1)\ln C$
(where 	all the $g_j$'s are bounded above by $C$). So $\phi_\mu$ is coercive and therefore
must have a (global) minimizer $\x_\mu\in\K$ with $g_j(\x_\mu)>0$ for every $j=1,\ldots,m$; and so $\nabla\phi_\mu(\x_\mu)=0$.
\end{proof}
Notice that $\phi_\mu$ may have several stationary points in $\K$. We now state our main result.
\begin{thm}
\label{th-main}
Let $\K$ in (\ref{setk}) be compact and let Assumption \ref{ass1} hold true.
For every fixed $\mu>0$, choose $\x_\mu\in\K$ to be an arbitrary stationary point of $\phi_\mu$ in $\K$.

Then every accumulation point $\x^*\in\K$ of such a sequence $(\x_\mu)\subset\K$ with $\mu\to 0$, is a global minimizer of $f$ on $\K$, and if $\nabla f(\x^*)\neq0$, $\x^*$ is
a KKT point of $\P$.
\end{thm}
\begin{proof}
Let $\x_\mu\in\K$ be a stationary point of $\phi_\mu$, which by Lemma \ref{lemma1} is guaranteed to exist. So
\begin{equation}
\label{cn}
\nabla\phi_\mu(\x_\mu)\,=\,\nabla f(\x_\mu)-\sum_{j=1}^m\frac{\mu}{g_j(\x_\mu)}
\nabla g_j(\x_\mu)\,=\,0.
\end{equation}
As $\mu\to 0$ and $\K$ is compact,
there exists $\x^*\in\K$ and a subsequence $(\mu_\ell)\subset\R_+$ 
such that $\x_{\mu_\ell}\to\x^*$ as $\ell\to\infty$.
We need consider two cases:

{\it Case when $g_j(\x^*)>0,\,\forall j=1,\ldots,m$.}
Then as $f$ and $g_j$ are continuously differentiable, $j=1,\ldots,m$, taking limit 
in (\ref{cn}) for the subsequence $(\mu_\ell)$, yields $\nabla f(\x^*)=0$
which, as $f$ is convex, implies that $\x^*$ is a global minimizer of $f$ on $\R^n$,
hence on $\K$.

{\it Case when $g_j(\x^*)=0$ for some $j\in\{1,\ldots,m\}$.}
Let $J:=\{j\,:\,g_j(\x^*)=0\}\neq\emptyset$. We next show that for every $j\in J$,
the sequence of ratios $(\mu/g_j(\x_{\mu_\ell})$, $\ell=1,\ldots$,  is bounded. Indeed
let $j\in J$ be fixed arbitrary. As Slater's condition holds,
let $\x_0\in\K$ be such that $g_j(\x_0)>0$ for all $j=1,\ldots,m$; then
$\langle \nabla g_j(\x^*),\x_0-\x^*\rangle>0$.
Indeed, as $\K$ is convex, $\langle \nabla g_j(\x^*),\x_0+\v-\x^*\rangle\geq0$ 
for all $\v$ in some small enough ball $\B(0,\rho)$ around the origin.
So if $\langle \nabla g_j(\x^*),\x_0-\x^*\rangle=0$ then 
$\langle \nabla g_j(\x^*),\v\rangle\geq0$ for all $\v\in \B(0,\rho)$, in contradiction
with $\nabla g_j(\x^*)\neq0$. Next,

\begin{eqnarray}
\label{aux}
\langle\nabla f(\x_{\mu_\ell}),\x_0-\x^*\rangle&=&
\underbrace{\sum_{k\not\in J}^m\frac{\mu_\ell}{g_k(\x_{\mu_\ell})}
\langle \nabla g_k(\x_{\mu_\ell}),\x_0-\x^*\rangle}_{A_\ell}\\
\nonumber
&&+
\underbrace{\sum_{k\in J}^m\frac{\mu_\ell}{g_k(\x_{\mu_\ell})}
\langle \nabla g_k(\x_{\mu_\ell}),\x_0-\x^*\rangle}_{B_\ell}\end{eqnarray}
Observe that in (\ref{aux}):
\begin{itemize}
\item Every term of the sum $B_\ell$ is nonnegative for sufficiently large $\ell$, say $\ell\geq\ell_0$, because $\x_{\mu_\ell}\to\x^*$ and
$\langle \nabla g_k(\x^*),\x_0-\x^*\rangle>0$ for all $k\in J$.
\item $A_\ell\to 0$ as $\ell\to\infty$ because $\mu_\ell\to0$ and $g_k(\x_{\mu_\ell})\to g_k(\x^*)>0$ for all $k\not\in J$.
\end{itemize}
Therefore $\vert A_\ell\vert\leq A$ for all sufficiently large $\ell$, say $\ell\geq\ell_1$, and so for every $j\in J$:
\[\langle\nabla f(\x_{\mu_\ell}),\x_0-\x^*\rangle+A
\geq\frac{\mu_\ell}{g_j(\x_{\mu_\ell})}\langle \nabla g_j(\x_{\mu_\ell}),\x_0-\x^*\rangle,\quad \ell\geq\ell_2:=\max[\ell_0,\ell_1],\]
which shows that for every $j\in J$, the nonnegative sequence $(\mu_\ell/g_j(\x_{\mu_\ell}))$, $\ell\geq\ell_2$, is bounded from above.

So take a subsequence (still denoted $(\mu_\ell)$, $\ell\in\N$, for convenience) such that
the ratios $\mu_\ell/g_j(\x_{\mu_\ell})$ converge for all $j\in J$, that is,
\[\lim_{\ell\to\infty}\,\frac{\mu_\ell}{g_j(\x_{\mu_\ell})}\,=\,\lambda_j\geq0,\qquad \forall\,j\in J,\]
and let $\lambda_j:=0$ for every $j\not\in J$, so that
$\lambda_jg_j(\x^*)=0$ for every $j=1,\ldots,m$. Taking limit in (\ref{cn}) as $\ell\to\infty$, yields:
\begin{equation}
\label{kkt}
\nabla f(\x^*)\,=\,\sum_{j=1}^m\lambda_j\,\nabla g_j(\x^*),
\end{equation}
which shows that $(\x^*,\lambda)\in\K\times \R^m_+$ is a KKT point for $\P$.
Finally, invoking Theorem 1, $\x^*$ is also a global minimizer of $\P$. 
\end{proof}
\subsection{Discussion}
The log-barrier function $\phi_\mu$ or its exponential variant 
$f+\mu\sum g_j^{-1}$ has become popular since the pioneer work
of Fiacco and McCormick \cite{fiacco1,fiacco2}, where it is assumed that $f$ and the $g_j$'s are
twice continuously differentiable, the $g_j$'s are concave\footnote{In fact as noted in \cite{fiacco1}, 
concavity of the $g_j$'s is merely a sufficient condition for the barrier function to be convex.}, Slater's condition holds, the set
$\K\cap\{\x\,:\,f(\x)\leq k\}$ is bounded for every finite $k$, and finally, the barrier function is strictly convex for every 
value of the parameter $\mu>0$. Under such conditions, the barrier function $f+\mu\sum g_j^{-1}$ 
has a unique minimizer
$\x_\mu$ for every $\mu>0$ and the sequence $(\x_\mu,(\mu/g_j(\x_\mu)^2)\subset\R^{n+m}$ converges to a Wolfe-dual feasible point.

In contrast, Theorem \ref{th-main} states that without assuming concavity of
the $g_j$'s, one may obtain a global minimizer of
$f$ on $\K$, by looking at {\it any} limit point of {\it any} sequence of 
stationary points $(\x_\mu)$, $\mu\to 0$, 
of the log-barrier function $\phi_\mu$ associated with 
a representation $(g_j)$ of $\K$, provided that the representation
satisfies the nondegeneracy condition (\ref{dege}). To us, this comes as a little surprise as 
the stationary points $(\x_\mu)$ are all inside $\K$, and there are examples
of convex sets $\K$ with a representation $(g_j)$ satisfying (\ref{dege})
and such that the level sets $\K_\a=\{\x : g_j(\x)\geq a_j\}$ with $a_j>0$, are not convex! (See Example \ref{ex1}.)
Even if $f$ is convex, the log-barrier function $\phi_\mu$ need not be convex;
for instance if $f$ is linear, $\nabla^2\phi_\mu=-\mu\sum_j\nabla^2 \ln g_j$,
and so if the  $g_j$'s are not log-concave then $\phi_\mu$ may not
be convex on $\K$ for every value of the parameter $\mu>0$.

\begin{ex}
\label{ex1}
Let $n=2$ and $\K_a:=\{\x\in\R^2\,:\,g(\x)\geq a\}$ with
$\x\mapsto g(\x):=4-((x_1+1)^2+x_2^2)((x_1-1)^2+x_2^2)$, 
with $a\in\R$. The set $\K_a$ is convex only for those values
of $a$ with $a\leq0$; see in Figure \ref{contour}. It is even disconnected for $a=4$.
\begin{figure}[ht]
\centering
%\begin{center}
\resizebox{0.8\textwidth}{!}
%{\includegraphics{contou}}
{\includegraphics{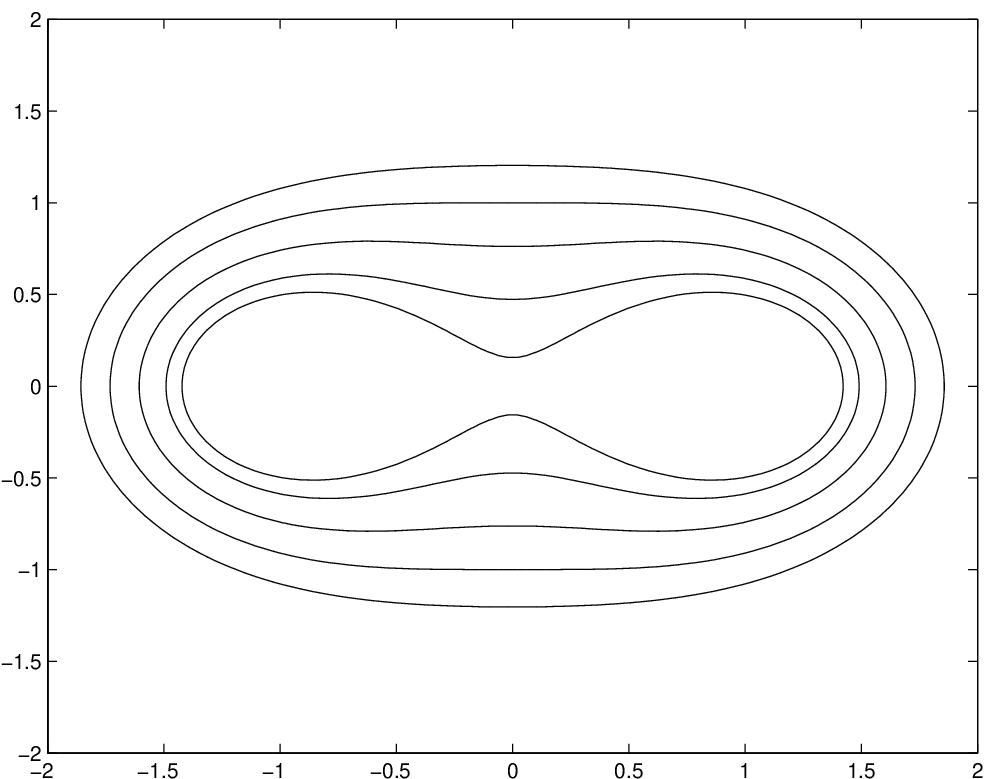}}
%\hspace{1cm}
%\includegraphics{sqrt2}}
\caption{Example \ref{ex1}: Level sets $\{\x:\,g(\x)=a\}$ for $a=2.95,2.5,1.5,0$ and $-2$}
\label{contour}
%\end{center}
\end{figure}
\end{ex}
We might want to consider a generic situation, that is, when
the set 
\[\K_\a:=\,\{\x\in\R^n\::\:g_j(\x)\geq a_j,\quad j=1,\ldots,m\},\]
is also convex for every positive vector $0\leq\a=(a_j)\in\R^m$. 
This in turn would imply that the $g_j$ are {\it quasiconcave}\footnote{Recall that
on a convex set $O\subset\R^n$, a function $f: O\to\R$ is quasiconvex if the level sets
$\{\x\,:\,f(\x)\leq r\}$ are convex for every $r\in\R$. A function $f: O\to\R$ is said to be 
quasiconcave if $-f$ is quasiconvex; see e.g. \cite{crouzeix}.} on $\K$.
In particular, if the nondegeneracy condition (\ref{dege}) holds on $\K$
and the $g_j$'s are twice differentiable,
then at most one eigenvalue of the Hessian $\nabla^2g_j$ (and hence $\nabla^2\ln g_j$) 
is possibly positive (i.e., $\ln g_j$ is {\it almost} concave). This is because
for every $\x\in\K$ with $g_j(\x)=0$, 
one has $\langle\v,\nabla^2 g_j(\x)\v\rangle\leq0$ for all $\v\in\nabla g_j(\x)^{\perp}$
(where $\nabla g_j(\x)^{\perp}:=\{\v:\langle\nabla g_j(\x),\v\rangle=0\}$).
However, 
even in this situation, the log-barrier function $\phi_\mu$ may not be convex.
On the other hand, $\ln g_j$ is "more" concave than $g_j$ on $\ite\K$ because
its Hessian $\nabla^2g_j$ satisfies
$g_j^2\nabla^2 \ln g_j=g_j \nabla^2g_j -\nabla g_j\,(\nabla g_j)^T$.
But still, $g_j$ might not be log-concave on $\ite\K$, and so $\phi_\mu$ may not be
convex at least for values of $\mu$ not too small (and for all values of $\mu$ if $f$ is linear).

\begin{ex}
\label{ex2}
Let $n=2$ and $\K:=\{\x: g(\x)\geq 0,\x\geq0\}$ with $\x\mapsto g(\x)=x_1x_2-1$.
The representation of $\K$ is not convex but the $g_j$'s are log-concave, and so
the log-barrier $\x\mapsto \phi_\mu(\x):=f\x)-\mu(\ln g(\x)-\ln x_1-\ln x_2)$ is convex.
\end{ex}
\begin{ex}
\label{ex3}
Let $n=2$ and $\K:=\{\x: g_1(\x)\geq 0;\,a-x_1\geq0;\,0\leq x_2\leq b\}$ with $\x\mapsto g_1(\x)=x_1/(\epsilon+x_2^2)$ with $\epsilon  >0$. 
The representation of $\K$ is not convex and $g_1$ is not log-concave.
If $f$ is linear and $\epsilon$ is small enough, the log-barrier 
\[\x\mapsto \phi_\mu(\x):=f(\x)-\mu(\ln x_1+\ln (a-x_1)-\ln(\epsilon+x_2^2)+\ln x_2+\ln (b-x_2))\]
is not convex for every value of $\mu>0$.
\end{ex}

\subsection*{Acknowledgement} The author wishes to thank two anonymous referees for pointing out a mistake and providing suggestions to improve 
the initial version of this paper.

\end{document}